\RequirePackage{fix-cm}
\PassOptionsToPackage{numbers,sort&compress,square}{natbib}
\documentclass[smallcondensed,natbib]{svjour3}     %
\ifdefined \pdfoutput
  \pdfoutput=1
\fi
\bibpunct[,]{[}{]}{,}{n}{,}{,}

\makeatletter
\if@twocolumn
  \renewcommand\normalsize{%
   \@setfontsize\normalsize\@xpt{12.5pt}%
   \abovedisplayskip=3 mm plus6pt minus 4pt
   \belowdisplayskip=3 mm plus6pt minus 4pt
   \abovedisplayshortskip=0.0 mm plus6pt
   \belowdisplayshortskip=2 mm plus4pt minus 4pt
   \let\@listi\@listI}%

  \renewcommand\small{%
   \@setfontsize\small{8.5pt}\@xpt
   \abovedisplayskip 8.5\p@ \@plus3\p@ \@minus4\p@
   \abovedisplayshortskip \z@ \@plus2\p@
   \belowdisplayshortskip 4\p@ \@plus2\p@ \@minus2\p@
   \def\@listi{\leftmargin\leftmargini
               \parsep 0\p@ \@plus1\p@ \@minus\p@
               \topsep 4\p@ \@plus2\p@ \@minus4\p@
               \itemsep0\p@}%
   \belowdisplayskip \abovedisplayskip}

\else
  \if@smallext
   \renewcommand\normalsize{%
   \@setfontsize\normalsize\@xpt\@xiipt
   \abovedisplayskip=3 mm plus6pt minus 4pt
   \belowdisplayskip=3 mm plus6pt minus 4pt
   \abovedisplayshortskip=0.0 mm plus6pt
   \belowdisplayshortskip=2 mm plus4pt minus 4pt
   \let\@listi\@listI}%

  \renewcommand\small{%
   \@setfontsize\small\@viiipt{9.5pt}%
   \abovedisplayskip 8.5\p@ \@plus3\p@ \@minus4\p@
   \abovedisplayshortskip \z@ \@plus2\p@
   \belowdisplayshortskip 4\p@ \@plus2\p@ \@minus2\p@
   \def\@listi{\leftmargin\leftmargini
               \parsep 0\p@ \@plus1\p@ \@minus\p@
               \topsep 4\p@ \@plus2\p@ \@minus4\p@
               \itemsep0\p@}%
   \belowdisplayskip \abovedisplayskip}
 \else
  \renewcommand\normalsize{%
   \@setfontsize\normalsize{9.5pt}{11.5pt}%
   \abovedisplayskip=3 mm plus6pt minus 4pt
   \belowdisplayskip=3 mm plus6pt minus 4pt
   \abovedisplayshortskip=0.0 mm plus6pt
   \belowdisplayshortskip=2 mm plus4pt minus 4pt
   \let\@listi\@listI}%

  \renewcommand\small{%
   \@setfontsize\small\@viiipt{9.25pt}%
   \abovedisplayskip 8.5\p@ \@plus3\p@ \@minus4\p@
   \abovedisplayshortskip \z@ \@plus2\p@
   \belowdisplayshortskip 4\p@ \@plus2\p@ \@minus2\p@
   \def\@listi{\leftmargin\leftmargini
               \parsep 0\p@ \@plus1\p@ \@minus\p@
               \topsep 4\p@ \@plus2\p@ \@minus4\p@
               \itemsep0\p@}%
   \belowdisplayskip \abovedisplayskip}
  \fi
\fi

\makeatother

\usepackage[T1]{fontenc}
\usepackage[english]{babel}
\usepackage[autostyle]{csquotes}

\usepackage{xcolor}
\usepackage[final]{graphicx}
\usepackage{fixltx2e}
\usepackage{bm}
\usepackage{microtype}
\usepackage{mathtools}
\usepackage{amssymb,amsmath,amsfonts}
\usepackage{leftidx}
\usepackage{ifluatex}
\usepackage{enumitem}
\usepackage{graphicx}
\usepackage{txfonts}

\usepackage{siunitx}
\usepackage{subcaption}

\usepackage[final,framed,numbered]{matlab-prettifier}
\makeatletter
\newcommand{\footnoteref}[1]{%
\ltx@ifpackageloaded{hyperref}{%
  \ifHy@hyperfootnotes%
    \hbox{\hyperref[#1]{%
            \@textsuperscript {\normalfont \ref*{#1}}}}%
  \else%
    \hbox{\@textsuperscript {\normalfont \ref*{#1}}}%
  \fi%
}{%
    \hbox{\@textsuperscript {\normalfont \ref{#1}}}%
 }%
}
\makeatother

\usepackage{doi}

\ifluatex
\providecommand{\versionortoday}{%
  \directlua{
    local f = io.popen('git describe --dirty')
    if f == nil then
    else
      tex.print('Version: ', f:read())
      f:close()
    end
  }
}
\else
\providecommand{\versionortoday}{\today}
\fi

\renewcommand{\makeheadbox}{}

\spnewtheorem{assumption}{Assumption}{\bfseries}{\rmfamily}
\spnewtheorem{algorithm}{Algorithm}{\bfseries}{\rmfamily}
\spnewtheorem{implementation}{Implementation}{\bfseries}{\rmfamily}

\definecolor{tol00}{HTML}{AA4499}
\definecolor{tol01}{HTML}{882255}
\definecolor{tol02}{HTML}{CC6677}
\definecolor{tol03}{HTML}{DDCC77}
\definecolor{tol04}{HTML}{999933}
\definecolor{tol05}{HTML}{117733}
\definecolor{tol06}{HTML}{44AA99}
\definecolor{tol07}{HTML}{88CCEE}
\definecolor{tol08}{HTML}{332288}

\colorlet{dotcolor}{tol00}
\colorlet{figcolor0}{tol00}
\colorlet{figcolor1}{tol04}
\colorlet{figcolor2}{tol08}

\newcommand{\bu}{{u}}
\newcommand{\bw}{{w}}
\newcommand{\bg}{{g}}

\newcommand{\bx}{\bm{x}}
\newcommand{\bp}{\bm{p}}
\newcommand{\balpha}{\bm{\alpha}}

\AtBeginDocument{\def\Re{\mathbb{R}}}
\AtBeginDocument{}

\DeclareMathOperator*{\argmin}{arg\,min}

\DeclareMathOperator*{\csch}{csch}
\DeclareMathOperator*{\sech}{sech}

\providecommand\given{}
\newcommand\SetSymbol[1][]{\nonscript\:#1\vert\nonscript\:\allowbreak}
\DeclarePairedDelimiterX\Set[1]\{\}{%
  \renewcommand\given{\SetSymbol[\delimsize]}
  #1
}

\journalname{myjournal}

\titlerunning{Solving 1D Conservation Laws Using Pontryagin's Minimum Principle}
\title{Solving 1D Conservation Laws Using Pontryagin's Minimum Principle
  \thanks{This document has been approved for public release; its
    distribution is unlimited.}
}

\author{Wei Kang \and Lucas~C. Wilcox}
\date{\versionortoday}
\institute{Wei Kang \and Lucas~C. Wilcox \at{} Department of Applied Mathematics,
  Naval Postgraduate School, Monterey, CA \\
  \email{\{wkang,lwilcox\}@nps.edu}
}

\begin{document}
\maketitle

\begin{abstract}
  This paper discusses a connection between scalar convex conservation
  laws and Pontryagin's minimum principle.
  For flux functions for which an associated optimal control problem
  can be found, a minimum value solution of the conservation law
  is proposed.
  For scalar space-independent convex conservation laws such a
  control problem exists and the minimum value solution of the
  conservation law is equivalent to the entropy solution.
  This can be seen as a generalization of the Lax--Oleinik formula to
  convex (not necessarily uniformly convex) flux functions.
  Using Pontryagin's minimum principle, an algorithm for finding the
  minimum value solution pointwise of scalar convex conservation laws
  is given.
  Numerical examples of approximating the solution of both
  space-dependent and space-independent conservation laws are provided
  to demonstrate the accuracy and applicability of the proposed
  algorithm.
  Furthermore, a MATLAB routine using Chebfun is provided (along with
  demonstration code on how to use it) to approximately solve scalar
  convex conservation laws with space-independent flux functions.
\end{abstract}


\section{Introduction}

Conservation laws arise naturally from continuum physics.
Scalar convex conservation laws have been studied extensively,
developing both theory and numerical methods, see for
example \citet{Dafermos2010,evans,Lax1973,Leveque1992}, and the
references there in.
Many of the approaches taken, both theoretical and numerical, use the
conservation law form of the problem.
However, for some applications, such as landform
evolution~\cite{Luke1972} and traffic flow~\cite{Newell1993},
it can be natural to consider the integrated unknown from the
conservation law.
This leads to a Hamilton--Jacobi equation.
This connection between conservation laws and Hamilton--Jacobi equations has
been exploited in developing numerical algorithms for Hamilton--Jacobi
equations~\cite{OsherSethian1988,OsherShu1991}, see
\citet[Chapter~5]{OsherFedkiw2003} for a description of this history.
The Hamilton--Jacobi equation
can be solved using the method of characteristics.
When characteristics cross a minimization principle is used by
\citet{Luke1972} and \citet{Newell1993} to select a unique physically
meaningful solution.
\citet{Daganzo2005} has shown that this minimization procedure
provides a stable solution.
Sharing the similar idea of Hamilton--Jacobi equation and
minimization, in this work we take a different approach in the study
of a general family of conservation laws and initial conditions.
Motivated by the the causality free method of solving
Hamilton--Jacobi--Bellman equations by
\citet{kangwilcox,KangWilcox2015arXiv}, we explore
the connections between the entropy solution of scalar convex
conservation laws and optimal control theory as well as the associated
Pontryagin's minimum principle, which leads to an efficient numerical
method for solving the conservation law.

\citet{corrias} proved that the entropy solution of a scalar convex
conservation law is the gradient of the viscosity solution of an
associated Hamilton--Jacobi equation.
This underlying relationship was also used to find the numerical solution of
conservation laws with large time steps~\citep{QiuShu2008}.
This paper extends this connection further and introduces optimal
control problems related to conservation laws.
When an associated optimal control problem can be found, which is
always the case for conservation laws with space-independent flux
functions, Pontryagin's minimum principle gives a set of necessary
conditions which can be used, along with the cost function, in an
algorithm to find a minimum value solution of the conservation law.
For conservation laws with space-independent flux functions, we show
that this minimum value solution is indeed the entropy solution.
This generalizes the Lax--Oleinik formula (see for example
\citet[Section~3.4.2]{evans}) to convex (not necessarily uniformly
convex) flux functions.

An alternative to the method proposed in this paper is the recent work of
\citet{DarbonOsher2016arXiv} and
\citet{ChowDarbonOsherYin2016report}, which solves the Hopf formula to find the
viscosity solution of Hamilton--Jacobi--Bellman equations. The algorithm is
causality free and it is effective in solving high dimensional problems.
The algorithm given by \citet{DarbonOsher2016arXiv}, like the one explored in
this paper, also has an interesting link to conservation laws.
For problems with a space-independent Hamiltonian and a convex initial
condition, the algorithm converges to not only the solution of the
Hamilton--Jacobi--Bellman equations, but also its gradient, which is a
entropy solution of the corresponding conservation law.

The necessary conditions from Pontryagin's minimum principle are a set of
boundary value problems that can be solved numerically using various techniques,
we use a spectral collocation method provided by the Chebfun software
package~\cite{chebfun}.
Although the solution of a conservation law is the gradient of the
value function of an optimal control, Pontryagin's minimum principle
contains this gradient as a part of its solution so that numerical
differentiation is not necessary.
For space-independent flux functions these necessary conditions
become algebraic equations, which can be approximately solved
using piecewise polynomial functions; again we use Chebfun in our
implementation of the algorithm.
For this case, we provide code for a numerical implementation of
solving scalar convex conservation laws with space-independent flux
functions pointwise.
The algorithm does not need a grid in space and time. It achieves high
accuracy even around shocks.
In both space-dependent and -independent
cases, minimizing the cost function provides a unique solution
of the control problem and its gradient provides a pointwise solution
of the conservation law.
We finish our discussion with numerical examples demonstrating
accuracy and applicability of the proposed algorithm where scalar convex
conservation laws, such as Burgers' equation, are solved and compared
to analytical or numerical solutions using other techniques.

\section{From Pontryagin's minimum principle to conservation law}
\label{sec_2}
In this section, we outline the underlying relationship between
Pontryagin's minimum principle (PMP) and the solution of a
conservation law.
Consider the conservation law
\begin{subequations}
\label{cons_law}
\begin{alignat}{2}
  \bu_t+{(F(x,\bu))}_x&=0   && \quad\text{in }\Re\times(0,T)\\
                   \bu&=\bg && \quad\text{on }\Re\times\{t=0\},
\end{alignat}
\end{subequations}
where the flux function $F:\Re\times\Re\to\Re$ and initial condition
$\bg:\Re\to\Re$ are given and $\bu:\Re\times(0,T)\to\Re$ is the
unknown, $\bu=\bu(x,t)$.
The associated Hamilton--Jacobi (HJ) equation has the following form
\begin{subequations}
\label{HJ}
\begin{alignat}{2}
  \bw_t+F(x,\bw_x)&=0 && \quad\text{in }\Re\times(0,T)\\
               \bw&=G && \quad\text{on }\Re\times\{t=0\},
\intertext{where the initial condition $G:\Re\to\Re$ is such that}
G'&=\bg &&\quad\text{almost everywhere in }\Re \label{G2g}
\end{alignat}
\end{subequations}
and $\bw:\Re\times(0,T)\to\Re$ is the
unknown, $\bw=\bw(x,t)$.
At any point where $\bw(x,t)$ has second order derivatives, we have
\begin{equation}
\label{u2w}
\bu(x,t)=\bw_x(x,t)
\end{equation}
is a solution of the conservation law~\eqref{cons_law}.
Given the initial condition, $(x,s)\in\Re\times[0,T]$,
consider a related problem of optimal control
\begin{subequations}
\label{opt_contr}
\begin{equation}
  \min_{\balpha} J_{x,s}[\balpha] = \min_{\balpha} \int_{s}^{T} L(\bx(r),\balpha(r))\;dr +G(\bx(T)),
\end{equation}
where the response of the system, $\bx:[s,T]\to\Re$, is subject to
\begin{equation}
\dot \bx(r)=\balpha (r), \quad \bx(s)=x.
\end{equation}
\end{subequations}
$\balpha:[0,T]\to A$ is a measurable function that represents the
control input; $A$ is a subset of $\Re$;
$J_{x,s}:A\to\Re$ is the cost function; and
$L:\Re\times A\to\Re$ is called the \emph{Lagrangian}.  Define the
\emph{Hamiltonian}, $H:\Re\times\Re\times A\to\Re$ as
\[H(x,p,\alpha)=p\alpha + L(x,\alpha)\]
and the \emph{value function}, $V:\Re\times[0,T]\to\Re$, as
\begin{equation}
  \label{value}
  V(x,s) = \inf_{\balpha}J_{x,s}[\balpha].
\end{equation}
Then, the associated Hamilton--Jacobi--Bellman (HJB) equation is
\begin{subequations}
\label{HJB}
\begin{alignat*}{2}
  V_s+\min_\alpha H(x,V_x,\alpha) &=0 && \quad\text{in }\Re\times(0,T)  \\
                                V&=G && \quad\text{on }\Re\times\{s=T\}.
\end{alignat*}
\end{subequations}
If we define $t=T-s$ and $\bw(x,t)=V(x,T-t)$ we can rewrite this HJB
equation as
\begin{subequations}
\label{HJ_b}
\begin{alignat}{2}
\bw_t(x,t)-\min_{\alpha} H(x,\bw_x,\alpha)&=0 && \quad\text{in }\Re\times(0,T)\\
        \bw&=G                              && \quad\text{on }\Re\times\{t=0\}.
\end{alignat}
\end{subequations}
Suppose the flux function and Hamiltonian are related such that
\begin{equation}
\label{F2H}
F(x,p)=-\min_\alpha H(x,p,\alpha)
\end{equation}
then the HJB equation~\eqref{HJ_b} is equivalent to the HJ
equation~\eqref{HJ}. Therefore, solving the optimal control
problem~\eqref{opt_contr} leads us to a solution of the HJ
equation~\eqref{HJ}; and then to a solution of the conservation
law~\eqref{cons_law} through~\eqref{u2w}.  We summarize this relation
in the following proposition.

\begin{proposition}
\label{prop0}
Let $V(x, s)$ be the value function defined in~\eqref{value} by the
optimal control~\eqref{opt_contr}. Suppose the flux function, $F$, in
the conservation law~\eqref{cons_law} and HJ equation~\eqref{HJ}
satisfies~\eqref{F2H} and suppose $G'(x)=\bg(x)$ almost
everywhere. Then the function $\bw(x,t)=V(x,T-t)$ satisfies the
HJ equation~\eqref{HJ} and the function
$\bu(x,t)=\bw_{x}(x,t)$ satisfies the conservation
law~\eqref{cons_law} at all points where the second order derivatives
of $\bw(x,t)$ exist.
\end{proposition}

A conservation law may have multiple solutions with non-smooth
properties such as shock and rarefaction waves. Uniqueness has been
proved in the literature for entropy solutions. The relationship
between $\bw(x,t)$ in Proposition~\ref{prop0} and the entropy solution
is addressed in Section~\ref{sec_3} for a family of conservation laws
with a convex flux functions. In general, we call $\bu (x,t)=\bw_x(x,t)$ the
\emph{minimum value solution} of the conservation law~\eqref{cons_law}
with respect to the Lagrangian $L(x,\alpha)$. Please note that the
minimum value solution is unique for a given Lagrangian at all points
where $\bw(x,t)$ admits the second order derivatives. This uniqueness
is due to the fact that the value function of a problem of optimal
control is unique.

The problem of finding the minimum value solution for a conservation
law boils down to solving the optimal control
problem~\eqref{opt_contr}.  Let $\balpha^\ast(x,p)$
be a solution of the following minimization
\begin{equation}
\label{alpha}
\min_{\alpha} H(x,p,\alpha).
\end{equation}
From PMP, an optimal trajectory of the control
problem~\eqref{opt_contr} satisfies the following necessary conditions
\begin{subequations}
\label{PMP0}
\begin{align}
\dot \bx (r) &= \balpha^\ast(\bx(r),\bp(r))\\
\dot \bp (r) &= -\frac{\partial H}{\partial x}\left(\bx(r),\bp(r),\balpha^{*}\left(\bx(r),\bp(r)\right)\right)\\
\bx(s)&=x\\
\bp(T)&=G'(\bx(T))
\end{align}
\end{subequations}
in which the costate, $\bp:[s,T]\to\Re$,
corresponds the gradient of the optimal cost, i.e.,
\[\bp(r)=V_x(\bx(r),r).\]
Because all functions do not explicitly depend on $r$, we can shift
$s$ to $0$ so that the boundary value problem~\eqref{PMP0} is
equivalent to
\begin{subequations}
\label{PMP}
\begin{align}
\dot \bx (r) &= \balpha^\ast(\bx(r),\bp(r))\label{PMPa}\\
\dot \bp (r) &= -\frac{\partial H}{\partial x}\left(\bx(r),\bp(r),\balpha^{*}\left(\bx(r),\bp(r)\right)\right) \label{PMPb}\\
\bx(0)&=x\\
\bp(t)&=G'(\bx(t))
\end{align}
\end{subequations}
where $t=T-s$. If $G'(x)$ does not exist at $a_k$, for
$k=1,2,\ldots,m$, an optimal trajectory may satisfy the following
equations
\begin{subequations}
\label{PMPbd}
\begin{align}
\dot \bx (r) &= \balpha^\ast(\bx(r),\bp(r))\label{PMPa2}\\
\dot \bp (r) &= -\frac{\partial H}{\partial x}\left(\bx(r),\bp(r),\balpha^{*}\left(\bx(r),\bp(r)\right)\right)\label{PMPb2}\\
\bx(0)&=x\\
\bx(t)&=a_k
\end{align}
\end{subequations}
for $k=1,2,\ldots, m$. Both~\eqref{PMP} and~\eqref{PMPbd} are
two-point boundary value problems (BVPs). If an optimal control problem
can be derived in which the Lagrangian $L(x,\alpha)$ and its
associated Hamiltonian that satisfies~\eqref{F2H}, then the following
is a high-level algorithm for finding the minimum value solution based
on the PMP\@.  Note that this algorithm finds both the minimum value
solution to the conservation law~\eqref{cons_law} and the solution to
the associated HJ equation~\eqref{HJ}.

\begin{algorithm}[Minimum value solution]
\label{minimum_value_solution}
\begin{description}[leftmargin=8em,labelindent=3em,style=nextline,itemsep=1ex]
  \item[\textbf{Step I}] For any given point $(x, t)$, find all
    solutions of the two-point BVPs~\eqref{PMP} and~\eqref{PMPbd}.
  \item[\textbf{Step II}] Among all the solutions in Step I, adopt the
    one, $(\bx,\bp)$, with the smallest cost,
    \[J_{x,t}^\ast = \int_{0}^{t} L\left(\bx(r),\balpha^{\ast}(\bx(r),\bp(r))\right)\;dr + G(\bx(t)).\]
  \item[\textbf{Step III}] Set \[\bw(x,t)=V(x,T-t)=J_{x,t}^\ast, \;\; \bu(x,t)=\bp(0).\]
\end{description}
\end{algorithm}

For the rest of the paper, we address the following questions that are
important to the algorithm.
\begin{itemize}
\item What is the relationship between the minimum value solution and
  the entropy solution?
\item How to find a Lagrangian $L(x,\alpha)$ that satisfies~\eqref{F2H}?
\item How to find all solutions of the two-point BVPs~\eqref{PMP}
  and~\eqref{PMPbd}?
\end{itemize}
In addition, several examples are shown in the following sections to
test the algorithm for conservation laws with various types of flux
functions.

\section{From viscosity solutions to entropy solutions}
\label{sec_3}
This section describes the connection between the minimum value
solution and the entropy for space-independent flux functions.
The following assumptions are made in several theorems that follow.
\begin{assumption}
\label{ass1}
The initial condition function $\bg (x)$, in the conservation
law~\eqref{cons_law}, is bounded in $\Re$ and it is continuous
everywhere except for a finite number of points,
$x=a_1,a_2, \ldots, a_m$.
\end{assumption}

\begin{assumption}
\label{ass2}
In the conservation law~\eqref{cons_law}, the flux function
$F(x,p)=F(p)$ is independent of $x$. In addition, $F$ is a
convex function satisfying
\begin{equation}
\label{F_nonlinear}
\lim_{|p|\to \infty} \frac{F(p)}{|p|} = +\infty.
\end{equation}
\end{assumption}

\begin{definition}
Given a function $F:\Re\to\Re$, its Legendre transform is
\[F^\ast (q)=\sup_{p} \left\{ pq-F(p)\right\} \quad \text{for } q\in \Re.\]
\end{definition}
For convex functions satisfying~\eqref{F_nonlinear}, the `$\sup$' in
the definition can be replaced by `$\max$.'
The following lemma formalizes when the Legendre transform of a convex
function is a convex function and when the transform is an involution.
\begin{lemma}[Convex duality~\cite{evans}]
\label{lm1}
Assume $F(p)$ satisfies Assumption~\ref{ass2}. Then
\begin{itemize}
  \item[(i)] the mapping $q\to F^\ast(q)$ is convex and
    \[\lim_{|q|\to \infty} \frac{F^\ast(q)}{|q|} = +\infty.\]
  \item[(ii)] Moreover \[F={(F^\ast)}^\ast.\]
\end{itemize}
\end{lemma}

In this section, we consider the conservation law
\begin{subequations}
\label{cons_law2}
\begin{alignat}{2}
  \bu_t+{(F(\bu))}_x&=0   && \quad\text{in }\Re\times(0,T)\\
                 \bu&=\bg && \quad\text{on }\Re\times\{t=0\},
\end{alignat}
\end{subequations}
in which $\bg$ and $F$ satisfy Assumptions~\ref{ass1}
and~\ref{ass2}. The associated HJ equation has the form
\begin{subequations}
\label{HJ2}
\begin{alignat}{2}
  \bw_t+F(\bw_x)&=0   &&\quad\text{in }\Re\times(0,T)\\
             \bw&=G   &&\quad\text{on }\Re\times\{t=0\}\\
              G'&=\bg &&\quad\text{almost everywhere in }\Re.
\end{alignat}
\end{subequations}
Following the idea in Section~\ref{sec_2}, the Hamiltonian of the
optimal control problem must satisfy~\eqref{F2H}.  Let
$F\circ(-1):\Re\to\Re$ represent the function
\[p \to F(-p).\]
Define the function $L:\Re\to\Re$ by
\begin{equation}
\label{F2L}
L(\alpha)={(F\circ (-1))}^\ast(\alpha)
         =\max_{p} \left\{ p\alpha -F(-p)\right\}.
\end{equation}
Given the initial condition, $(x,s)\in\Re\times[0,T]$,
the associate problem of optimal control has the form
\begin{subequations}
\label{opt_contr2}
\begin{equation}
\min_{\balpha} J_{x,s}[\balpha] = \min_{\balpha} \int_{s}^{T} L(\balpha(r))\;dr +G(\bx(T))
\end{equation}
subject to
\begin{equation}
\dot \bx(r)=\balpha (r), \quad \bx(s)=x.
\end{equation}
\end{subequations}
The Hamiltonian is independent of $x$, specifically
\[H(p,\alpha)=p\alpha+L(\alpha).\]
The following lemma shows that the requirement~\eqref{F2H} is fulfilled.
\begin{lemma}
\label{lm2}
Suppose Assumption~\ref{ass2} holds. Then
\[F(p)=-\min_{\alpha}H(p,\alpha).\]
\end{lemma}
\begin{proof} From Lemma~\ref{lm1}, we have
\[F\circ (-1) = L^\ast.\]
More specifically, for any $p\in \Re$
\[
\begin{split}
  F(p) &= L^\ast (-p)\\
       &=  \max_{\alpha} \left\{ (-p)\alpha -L(\alpha)\right\}\\
       &= -\min_{\alpha} \left\{   p \alpha +L(\alpha)\right\}\\
       &= -\min_{\alpha} H(p,\alpha).
\end{split}
\]
\qed{}
\end{proof}

The following lemma will be used to simplify the form of the dynamics
given in the PMP\@.
\begin{lemma}
\label{lm3}
Suppose $F \in C^1(\Re)$ satisfies Assumption~\ref{ass2}. Then the function
\begin{equation*}
\label{alpha3}
\balpha^\ast (p)=-F^\prime (p), \; \;\; p\in  \Re
\end{equation*}
is the unique solution of the following minimization
\begin{equation*}
\label{alpha2}
\min_{\alpha}H(p,\alpha).
\end{equation*}
\end{lemma}
\begin{proof}
From the definition~\eqref{F2L},
\[L(\alpha) = \max_{\lambda}\left\{ \lambda\alpha - F(-\lambda)\right\}\]
for all $\alpha\in\Re$. Equivalently
\[L(\alpha) = \lambda\alpha - F(-\lambda)\] where $\lambda$ is a number
satisfying
\begin{equation}
\label{ad1_v1}
\alpha + F^\prime (-\lambda)=0.
\end{equation}
For a given $\alpha$, the value of $\lambda$ satisfying (\ref{ad1_v1}) may not
be unique. However, the minimum value, $L(\alpha)$, is unique because $
\lambda\alpha - F(-\lambda)$ is convex. Define 
\[
\begin{split}
\bar H(p,\lambda) &= H(p, -F^\prime (\lambda)) \\
&=-(p+\lambda)F^\prime (-\lambda)-F(-\lambda).
\end{split}
 \]
Assumption~\ref{ass2} implies that  $F^\prime (-\lambda)$ is monotone and
unbounded. Therefore, $\alpha$ minimizes $H(p,\alpha)$ if and only if there is a
number $\lambda$ satisfying (\ref{ad1_v1}) that minimizes $\bar H(p,\lambda)$. To
minimize $\bar H(p,\lambda)$, we consider
\[
\begin{split}
\bar H(p,\lambda) & = -(p+\lambda)F^\prime (-\lambda)-F(-\lambda) +F(p) - F(p)\\
&= -(p+\lambda)F^\prime (-\lambda)+(p+\lambda)F^\prime (\xi)- F(p)\\
&=-(p+\lambda)(F^\prime (-\lambda)-F^\prime (\xi))-F(p)
\end{split}
\]
where $\xi$ is a number between $p$ and $-\lambda$. If $p+\lambda >0$, then $p>
\xi >-\lambda$. We know that $F^\prime (\cdot)$ is nondecreasing. Therefore,
\[ -(p+\lambda)(F^\prime (-\lambda)-F^\prime(\xi)) \geq 0\]
Similarly, we can prove the same inequality if $p+\lambda < 0$. Therefore,
\[\bar H(p,\lambda ) \geq -F(p) = \bar H(p,-p),\]
i.e., $\lambda = -p$ minimizes $\bar H(p,\lambda)$. Its minimum value is $-F(p)$.
Therefore, $\alpha^\ast=-F^\prime (p)$ minimizes $H(p,\alpha)$.

To prove that $\alpha^\ast=-F^\prime (p)$ is the unique function that minimizes
$H(p,\alpha)$, let us assume that $\alpha_1=-F^\prime(\lambda_1)$ minimizes
$H(p,\alpha)$ and $\bar H(p,\lambda)$. Then
\[ \bar H(p,\lambda_1) = -F(p),\]
i.e.,
\[-(p+\lambda_1)F^\prime (-\lambda_1)-F(-\lambda_1) + F(p)=0.\]
Equivalently,
\[F(p)=F(-\lambda_1) +(p+\lambda_1)F^\prime (-\lambda_1).\]
Because $F(\xi)$ is convex, its curve cannot lie below its tangent line. Therefore,
\[F(\xi)=F(-\lambda_1) +(\xi+\lambda_1)F^\prime (-\lambda_1)\]
for all $\xi$ between $\lambda_1$ and $p$. So,
\[ F^\prime (\xi) = F^\prime(-\lambda_1).\]
Let $\xi = p$, then $-F^\prime (p) = -F^\prime (\lambda_1) =\alpha_1 $. This
implies that $\alpha^\ast = -F^\prime (p)$ is the unique function that minimizes
$H(p,\alpha)$.
\qed{}
\end{proof}

Using Lemma~\ref{lm3}, the Hamilton
dynamics~\eqref{PMPa}--\eqref{PMPb} (and~\eqref{PMPa2}--\eqref{PMPb2})
for the control problem~\eqref{opt_contr2} is simplified to
\begin{subequations}
\label{PMP2}
\begin{align}
\dot \bx (r) &= \balpha^\ast(\bp(r))=-F^\prime (\bp(r))\\
\dot \bp (r) &= 0.
\end{align}
\end{subequations}
This implies that $\bp$ is a constant and the characteristic is a
straight line
\begin{alignat*}{2}
\bp (r) &\equiv p  &&\quad \text{(a constant)}\\
\bx (r) &= x-F^\prime (p) r.
\end{alignat*}
For this to be a solution of one of the two-point BVPs~\eqref{PMP}
and~\eqref{PMPbd}, $p$ must satisfy at least one of the following
equations
\begin{subequations}
\label{characteristic}
\begin{alignat}{2}
  p&=G^\prime (x-F^\prime (p)t)  \\
a_k&=x-F^\prime (p)t &&\quad \text{ for } 1\leq k\leq m,
\end{alignat}
\end{subequations}
where $a_k$'s are the points at which $G^\prime (x)$, or $\bg (x)$, is
discontinuous. Thanks to Assumption~\ref{ass2}, the two-point boundary
value problem (BVP) of
PMP boils down to the algebraic equations~\eqref{characteristic}. This
is fundamentally different from the PMP in Section~\ref{sec_2}, where
the problem is defined using differential equations. Under some
conditions which will be addressed later, all solutions of the
algebraic equations~\eqref{characteristic} can be found.
\begin{remark}
From the Hamiltonian dynamics~\eqref{PMP2}, the optimal trajectory is
a line \[x+\alpha r\] where $\alpha = \balpha (p)$ is a constant. The
problem of optimal control~\eqref{opt_contr2} is equivalent to
\begin{equation}
\label{eqtmp}
\min_{\alpha}\left\{ L(\alpha)t+G(x+\alpha t)\right\}.
\end{equation}
Define
\[y=x+\alpha t\]
then minimization problem~\eqref{eqtmp}, for $t>0$, is transformed to
\[\min_{y}\left\{tL\left(\frac{y-x}{t}\right)+G(y)\right\}.\]
This is the Hopf--Lax formula~\cite{evans}.
\end{remark}

Now, let us address the issue of discontinuity in solutions. Instead
of a classic smooth solution, we consider the entropy solution of
the conservation law~\eqref{cons_law2}. The problem is closely related
to the viscosity solution of associated HJ equation. Suppose that the
control input of the optimal control problem~\eqref{opt_contr2} is
bounded and measurable, i.e., the set of possible control inputs is
\begin{equation}
\label{inputspace}
\Set*{ \balpha: [0, T] \rightarrow [A_1, A_2] \given \balpha
\text{ is measurable}}.
\end{equation}
\begin{theorem}[\cite{evans}]
\label{thm1}
Suppose $\bg$ satisfies Assumption~\ref{ass1} and $G$
is bounded. Further, suppose $F \in C^1(\Re)$ satisfies
Assumption~\ref{ass2}. Let $V(x,s)$ be the value function of the
optimal control problem~\eqref{opt_contr2} with bounded control
inputs and define \[\bw(x,t) = V(x,T-t).\]
Then, $\bw(x,t)$ is the viscosity solution of the initial value
problem~\eqref{HJ2}.
\end{theorem}
It is proved by \citet{corrias} that the entropy solution of a
conservation law is the gradient of the viscosity solution of the HJ
equation.
\begin{theorem}[\cite{corrias}]
\label{thm2}
Suppose $\bg$ satisfies Assumption~\ref{ass1}. Suppose
$F\in C^1(\Re)$ satisfies Assumption~\ref{ass2}. If
$\bw \in W^{1,\infty}(\Re \times (0, T])$ is the (unique)
viscosity solution of the HJ equation~\eqref{HJ2}, then $\bu = \bw_x$
is the (unique) entropy solution of the conservation
law~\eqref{cons_law2}.
\end{theorem}
Theorems~\ref{thm1} and~\ref{thm2} almost bridge the HJ equation and
the conservation law except that the control input is required to be
bounded in~\eqref{inputspace}. In fact, it can be guaranteed that
\eqref{inputspace} holds true for the optimal control
problem~\eqref{opt_contr2}.
\begin{proposition}
\label{prop1}
Suppose that Assumptions~\ref{ass1} and~\ref{ass2} hold true. Then the
value of optimal control, $\alpha^\ast$, for the optimal control
problem~\eqref{opt_contr2} as a function of $x$ is bounded in $\Re$.
\end{proposition}
\begin{proof}
The result is trivially true if $F$ is $C^1$. However, it can be
proved without this smoothness assumption. From the relation of the
initial conditions~\eqref{G2g} and Assumption~\ref{ass1}, we know that
$G$ is Lipschitz. Let $C$ be its Lipschitz constant. Given $(x,s)$, the
optimal control is a constant function,
$\balpha^\ast \equiv \alpha^\ast$ for some $\alpha^\ast\in \Re$. The
corresponding optimal cost value is
\[J_{x,s}^\ast=\min_{\alpha} \{ L(\alpha)(T-s) +G(x+(T-s)\alpha)\}.\]
The optimal value is less than or equal to the value at $\alpha=1$, i.e.,
\[L(\alpha^\ast)(T-s) +G(x+(T-s)\alpha^\ast)\leq L(1)(T-s) +G(x+(T-s)).\]
Because $G$ is Lipschitz
\[
  \begin{split}
L(\alpha^\ast)(T-s) +G(x)-C(T-s)|\alpha^\ast|
 &\leq L(\alpha^\ast)(T-s) +G(x+(T-s)\alpha^\ast)\\
 &\leq L(1)(T-s) +G(x+(T-s))\\
 &\leq L(1)(T-s) +G(x)+C(T-s).
\end{split}
\]
Therefore,
\[L(\alpha^\ast)-C|\alpha^\ast|\leq  L(1) +C\]
or equivalently
\[\frac{L(\alpha^\ast)}{|\alpha^\ast|} \leq \frac{L(1)}{|\alpha^\ast|}
  +\frac{C}{|\alpha^\ast|} +C.\]
Therefore, $\alpha^\ast$ must be bounded because of the following
property of $L(\alpha)$
\[\lim_{|\alpha| \rightarrow \infty}
  \frac{L(\alpha)}{|\alpha|}=+\infty.\]
\qed{}
\end{proof}
The following section contains a causality free algorithm for the
computation of $u$ at any point $(x, t)$.

\section{A simplified numerical algorithm}
\label{sec_4}
Let us consider a conservation law satisfying Assumptions~\ref{ass1}
and~\ref{ass2}. It is proved in Section~\ref{sec_3} that its
entropy solution is the same as the minimum value solution with
respect to an associated Lagrangian. In addition, the characteristics
resulting from the PMP are straight lines. In this case,
Algorithm~\ref{minimum_value_solution} can be significantly simplified. In the following, 
\[L={(F\circ (-1))}^\ast.\]
If an explicit expression of this function is not derived, its value in (\ref{cost_alg}) can be computed using the following formula
\[L(-F^\prime(p)) = pF^\prime (p) - F(p).\]
because any $p$ satisfying $\alpha = -F^\prime(p)$ maximizes 
\[ \max_{p} \left\{ -p\alpha -F(p)\right\},\]
which is equivalent to (\ref{F2L}). 

\begin{algorithm}[Entropy solution]
\label{al_entropy_solution}
This algorithm is for conservation laws satisfying
Assumptions~\ref{ass1} and~\ref{ass2}.
\begin{description}[leftmargin=8em,labelindent=3em,style=nextline,itemsep=1ex]
\item[\textbf{Step I}]
Given any point $(x,t)$, find all values of $p$ that satisfy at
least one of the algebraic equations~\eqref{characteristic}, i.e.,
\begin{subequations}
\label{characteristic2}
\begin{alignat}{2}
  p&=\bg (x-F^\prime (p)t) \label{characteristic2a}\\
\intertext{or}
a_k&=x-F^\prime (p)t, &&\quad \text{ for } 1\leq k\leq m. \label{characteristic2b}
\end{alignat}
\end{subequations}
\item[\textbf{Step II}] Among all the solutions in Step I, adopt the one with
the smallest value,
\begin{equation}
  \label{cost_alg}
J^\ast = L(-F^\prime(p))t +G(x-F^\prime (p) t)
\end{equation}
or equivalently
\[J^\ast = (pF^\prime (p) - F(p))t +G(x-F^\prime (p) t)\]
\item[\textbf{Step III}] Set
\[
\bu(x,t)=p.
\]
\end{description}
\end{algorithm}

\begin{remark}
Step~I is critical in this algorithm. Given a general bounded initial
function, $\bg$, the algebraic equations~\eqref{characteristic2} may
have multiple solutions.  In the examples below we use
piecewise-polynomial approximation~\cite{chebfun_piecewise} (and
trigonometric approximation~\cite{chebfun_periodic} for periodic
functions) via Chebfun\footnote{Chebfun is a
  MATLAB~\cite{matlab} package that allows symbolic-like manipulation
  of functions at numerics speed using Chebyshev and Fourier
  series~\cite{chebfun}.}~\cite{chebfun} and its \lstinline|roots|
command to approximate the solutions of the algebraic
equations~\eqref{characteristic2}.
In the piecewise-polynomial case, the Chebfun command
\lstinline|roots| uses the algorithm described in \citet{boyd}.
Basically, it subdivides the interval into small pieces (based on the
number of terms in the approximations) and uses the eigenvalues of the
colleague matrix associated with the approximation on each
interval~\cite{specht,good}.  For further discussion of root finding
using Chebfun see \citet[Chaper~18]{Trefethen2012}.  For the
polynomial and trigonometric approximations used in this paper, we
uses Chebfun's adaptive procedure to find the number of terms
automatically to achieve roughly 15 digits of relative accuracy.
\end{remark}

Here is an implementation of Algorithm~\ref{al_entropy_solution} in
MATLAB~\cite{matlab} using Chebfun.
\begin{implementation}[Entropy solution for convex flux functions]\label{pdeccl}
\begin{lstlisting}
function u = pdeccl(Fp, L, g, G, d, a, x, t)
  %
  cFp = chebfun(Fp, d, 'splitting', 'on'); `\label{es_s1_eq2_begin}`
  q = [];
  for k = 1:length(a)
    hk = a(k) - x + t*cFp;
    q = [q, roots(hk)'];
  end `\label{es_s1_eq2_end}`

  dom = unique([d, q]); `\label{es_s1_eq1_begin}`
  h = chebfun(@(p) g(x - t*Fp(p)) - p, dom, 'splitting', 'on');
  p = [roots(h); q(:)]; `\label{es_s1_eq1_end}`

  %
  J = @(p) L(-Fp(p))*t + G(x-Fp(p)*t); `\label{es_s2_begin}`
  [~, idx] = min(J(p)); `\label{es_s2_end}`

  %
  u = p(idx); `\label{es_s3}`
end
\end{lstlisting}
Here, \lstinline|Fp|, \lstinline|L|, \lstinline|g|, and \lstinline|G|
are the functions $F'$, $L$, $g$, and $G$, respectively.  The array
\lstinline|d| is the domain of definition (i.e., the range of $g$
including breakpoints where $F'$ might not be smooth) that should be
used with Chebfun and the array \lstinline|a| contains the points at
which the initial condition $g$ is discontinuous.  The pair
(\lstinline|x|,\lstinline|t|) is the point $(x,t)$ at which the
solution $u$, returned as \lstinline|u|, of the conservation law is
computed.  If \lstinline|g| is a compactly supported
chebfun\footnote{Chebfun with a capital C is the name of the software
package, chebfun with a lowercase c is a single variable function
defined on an interval created with the package~\cite{chebfun}.}
(e.g., \lstinline|g = chebfun({0,1,0}, [-1,0,1,3])|) and $F'$ is
smooth then \lstinline|G|, \lstinline|d|, and \lstinline|a| can be
computed with:
\begin{lstlisting}[numbers=none,frame=none]
  G = cumsum(g);
  d = minandmax(g)';
  a = g.ends(abs(jump(g, g.ends)) > 10*vscale(g)*eps);
\end{lstlisting}
In this implementation, the Step~I consists of
lines~\ref{es_s1_eq2_begin}--\ref{es_s1_eq2_end} which solve the
algebraic equations~\eqref{characteristic2b} and
lines~\ref{es_s1_eq1_begin}--\ref{es_s1_eq1_end} which find all the
solutions of \eqref{characteristic2a}.
\end{implementation}

This algorithm has several advantages. In the presence of shock waves,
finding the unique entropy solution is simple because the minimum
value solution in Step~II is unique. Furthermore, unlike the
traditional method of characteristics for conservation laws, the
computation does not explicitly use the Rankine--Hugoniot condition to
calculate the behavior of the shock curves.

The algorithm is not based on interpolating and differentiation on a
spacial grid and thus can avoid the Gibbs--Wilbraham phenomenon when
calculating the solution at a point.  We say that the algorithm is
causality free or pointwise, meaning that the value of $u(x,t)$ is
computed without using the value of $u$ at any other points.  An
advantage of causality free algorithms is their perfect
parallelism. If the solution is to be computed at a large number of
points, the computation is embarrassingly parallel.

Also because of the causality free property, the error does not
propagate in space. The accuracy can be kept at the same level
throughout a region and is based on the accuracy approximating the
algebraic equations~\eqref{characteristic2} and finding their roots.

\begin{remark}
If we assume $t>0$ and that $F$ is uniformly
convex, i.e., $F^{\prime\prime}(p) > \gamma >0$, then the algebraic
equations~\eqref{characteristic2} can be derived from the
Lax--Oleinik formula in \citet{evans} given here as
\begin{subequations}
\begin{align}
u(x,t)&={(F^\prime)}^{-1}\left(\frac{x-\xi(x,t)}{t}\right)\label{LO1} \\
\intertext{where}
\xi(x,t) &= \argmin_{\xi} \left\{ tF^\ast\left(\frac{x-\xi}{t}\right)+G(\xi)\right\}.\label{LO2}
\end{align}
\end{subequations}
The critical points of the minimization problem in \eqref{LO2} must
satisfy one of the following equations
\begin{subequations}
\label{LO3}
\begin{alignat}{2}
{(F^\ast)}^\prime\left(\frac{x-\xi}{t}\right)-\bg(\xi)&=0\\
\intertext{or}
   \xi &= a_k &&\quad \text{ for } 1\leq k \leq m.
\end{alignat}
\end{subequations}
It can be proved that ${(F^\ast)}^\prime = {(F^\prime)}^{-1}$.
Therefore, if we define
\[\xi = x-F^\prime (p) t\]
the minimizations~\eqref{LO2} and \eqref{cost_alg} are equivalent.
Further, the equations for the critical points \eqref{LO3} are
transformed to \eqref{characteristic2}.
\end{remark}

\section{Examples}

In the following, we illustrate
Algorithms~\ref{minimum_value_solution} and~\ref{al_entropy_solution}
using several examples with different types of flux functions,
including both space-dependent and space-independent.
A specific thing to note is the lack of Gibbs--Wilbraham oscillations
in all of the approximate solutions given in the figures below, even
in the presence of solutions with many shocks.
All examples are computed using MATLAB 8.4.0.150421 (R2014b) and
Chebfun 5.3.0 with double floating-point precision on an Apple MacBook
Pro (Retina, 15-inch, Early 2013) with a 2.7 GHz Intel Core i7 central
processing unit and 16 GB of 1600 MHz DDR3 random-access memory
running OS~X 10.10.5.

\begin{remark}
  In the examples below we provide timings of Implementation~\ref{pdeccl} of the
  causality-free Algorithm~\ref{al_entropy_solution}.  Timings are
  also given for the second-order Lax--Wendroff finite volume
  method with the van Leer limiter from
  Clawpack\footnote{\label{foot:clawver}More specifically,
    we used Clawpack \texttt{v5.3.1-11-geb31727} from
    \url{https://github.com/clawpack/clawpack} with the Chapter~11
    examples in the git repository \url{https://github.com/clawpack/apps}
    (git commit
  \texttt{ba557b49852377c05192d48289b3fbc8fea0f52e}).}~\cite{clawpack,LeVeque2002}
  when it is compared with Implementation~\ref{pdeccl}.  These numbers show the
  current performance of our implementation.  It should be noted that no
  performance tuning has been done by the authors for either code.
  In our experience if high-accuracy and/or the long time solution is desired at
  a few points in space and time then Implementation~\ref{pdeccl} is a
  competitive method.  If lower-accuracy is okay and the solution is desired on
  a dense space-time grid then grid based methods, such as the finite-volume
  method, will be more competitive than Implementation~\ref{pdeccl}.
\end{remark}

\begin{example}\label{ex:burgers_const}
We start by solving Burgers' equation
\begin{alignat*}{2}
  \bu_t+{\left( \frac{u^{2}}{2}\right)}_x&=0 && \quad\text{in }\Re\times(0,T)\\
                                \bu&=
  \begin{cases}
1, & 0\leq x \leq 1\\
0, & \text{otherwise}
    \end{cases}
 && \quad\text{on }\Re\times\{t=0\}.
\end{alignat*}
It has a known solution given in \citet[Example~3 of
Chapter~3.4]{evans}
as the following
\begin{align*}
\text{for } 0&\leq t \leq 2  & \text{and for } t &\geq 2 \\
\bu (x,t)&= \begin{cases}
0, & x<0\\
\frac{x}{t}, &0< x<t\\
1, & t<x<1+\frac{t}{2}\\
0, &x>1+\frac{t}{2}
\end{cases}
     &
\bu (x,t)&= \begin{cases}
0, & x<0\\
\frac{x}{t}, &0< x<{(2t)}^{1/2}\\
0, &x>{(2t)}^{1/2}.
\end{cases}
\end{align*}
The solution has shock wave which travel at two different speeds for
$t<2$ and $t>2$, respectively.  It also has a rarefaction wave for
$t<2$. In the computation, we do not need any information about the
shock wave speed.

The causality-free Algorithm~\ref{al_entropy_solution} is applied
using Implementation~\ref{pdeccl}.  Note that for this simple case, a
piecewise constant initial condition for Burgers' equation, the
equations for the critical points~\eqref{characteristic2} can be
solved directly and Chebfun is not required.  However, we go ahead and
use Chebfun anyways to benchmark the implementation.
The maximum pointwise error (i.e., $l^{\infty}$) for the solution $u$
on a uniform grid of $100 \times 100$ points covering the domain $[-1, 3]
\times [0.1, 4]$ (the point within \num{2.2204e-15} of the shock was
excluded) is $\num{2.2204e-16}$.  For this error calculation, the
solution is computed in parallel using a MATLAB \lstinline|parfor|
loop at the rate of $78$ points per second (excluding the time for
starting the parallel pool).  Here \lstinline|g| and \lstinline|G| are
chebfuns, the code would be faster if these were simple MATLAB
functions.

Implementation~\ref{pdeccl} provides a pointwise solution that we
choose to feed back into Chebfun to build piecewise continuous
approximations to the solution shown in
Figure~\ref{fig:burgers_const}.
\begin{figure}
  \begin{minipage}[b]{.49\linewidth}
    \centering\includegraphics{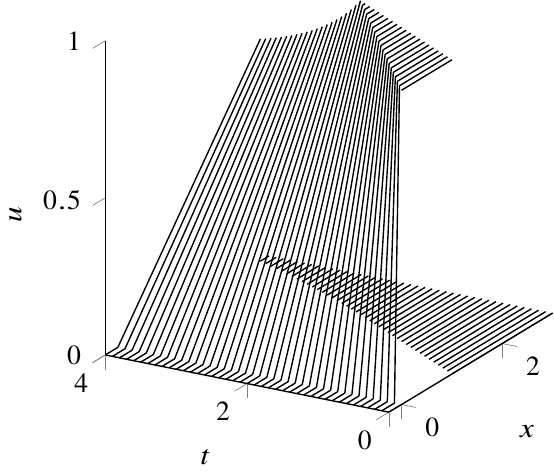}
    \subcaption{}\label{fig:1a}
  \end{minipage}%
  \hfill
  \begin{minipage}[b]{.49\linewidth}
    \centering\includegraphics{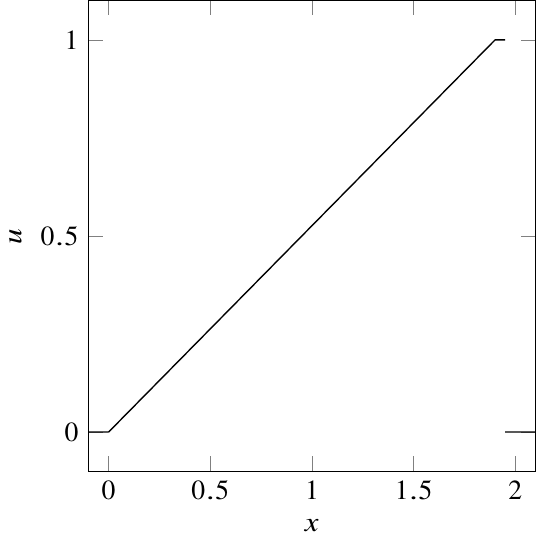}
    \subcaption{}\label{fig:1b}
  \end{minipage}
  \caption{The solution of Burgers' equation for the piecewise
    constant initial condition given in
    Example~\ref{ex:burgers_const}.
    Here (\subref{fig:1a}) is a waterfall plot of the solution and
    (\subref{fig:1b}) is the solution for
    $t=1.9$.}\label{fig:burgers_const}
\end{figure}
The code used to generate this figure is given in
Appendix~\ref{a:driver}.

\end{example}

\begin{example}\label{ex:burgers_sine}
We solve Burgers' equation
\begin{alignat*}{2}
  \bu_t+{\left( \frac{u^{2}}{2}\right)}_x&=0 && \quad\text{in }\Re\times(0,T)\\
                                \bu&=1+\sin(\pi x)
 && \quad\text{on }\Re\times\{t=0\},
\end{alignat*}
using Implementation~\ref{pdeccl} and compare the result against the
traditional method of characteristics for conservation laws.  The
maximum pointwise error (i.e., $l^{\infty}$) for the solution $u$ on
a uniform grid of $80 \times 80$ points covering the domain $[0, 4] \times
[0.1, 0.8]$ is $\num{1.2212e-14}$.  For this error calculation, the
solution is computed in parallel using a MATLAB \lstinline|parfor|
loop at the rate of $188$ points per second (excluding the time for
starting the parallel pool).  Note, here we use MATLAB functions for
\lstinline|g| and \lstinline|G|.

As in the previous example, we give the pointwise solution function
to Chebfun and build piecewise continuous approximations to the
solution shown in Figure~\ref{fig:burgers_sine}.
\begin{figure}
  \begin{minipage}[b]{.49\linewidth}
    \centering\includegraphics{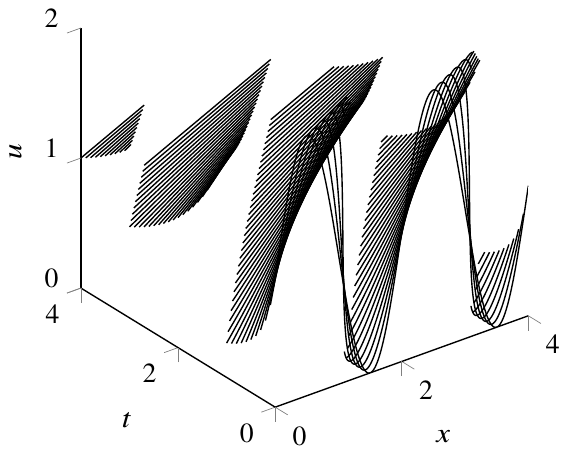}
    \subcaption{}\label{fig:2a}
  \end{minipage}%
  \hfill
  \begin{minipage}[b]{.49\linewidth}
    \centering\includegraphics{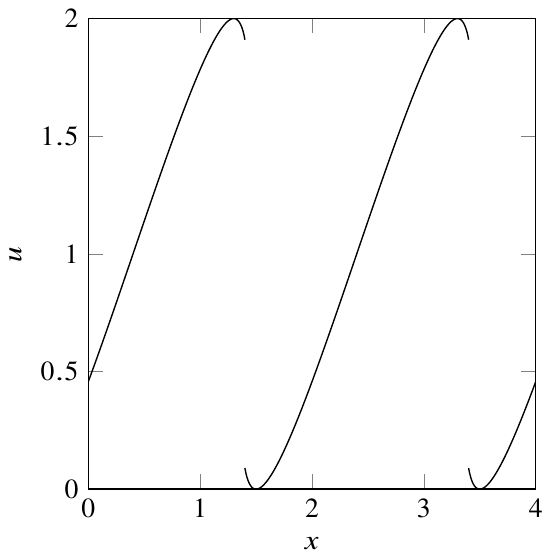}
    \subcaption{}\label{fig:2b}
  \end{minipage}
  \caption{The solution of Burgers' equation for the sinusoidal
    initial condition given in Example~\ref{ex:burgers_sine}.
    Here (\subref{fig:2a}) is a waterfall plot of the solution and
    (\subref{fig:2b}) is the solution for
    $t=0.4$.}\label{fig:burgers_sine}
\end{figure}

\end{example}

\begin{example}\label{ex:burgers_nwave}
We solve Burgers' equation with a compactly supported initial condition
\begin{alignat*}{2}
  \bu_t+{\left( \frac{u^{2}}{2}\right)}_x&=0 && \quad\text{in }\Re\times(0,T)\\
                                \bu&=\begin{cases}
(\cos(x)+1)\left(2\sin(3x)+\cos(2x)+\frac{1}{5}\right), & -\pi \leq x \leq \pi\\
0, & \text{otherwise}
\end{cases}
 && \quad\text{on }\Re\times\{t=0\},
\end{alignat*}
which is an example of N-wave decay given in Chapter~11.5 of \citet{LeVeque2002}.
As in the previous example, we give the pointwise solution function
using Implementation~\ref{pdeccl}
to Chebfun and build piecewise continuous approximations to the
solution shown in Figure~\ref{fig:burgers_nwave}.  For verification,
this figure also presents a solution to this problem using a
second-order Lax--Wendroff finite volume method with the van Leer
limiter from
Clawpack\footnoteref{foot:clawver}~\cite{clawpack,LeVeque2002}
on a grid from $[-8,8]$ of 1000 cells, which took \num{0.021} seconds to
compute to $t=1$.
\begin{figure}
  \begin{minipage}[b]{.49\linewidth}
    \centering\includegraphics{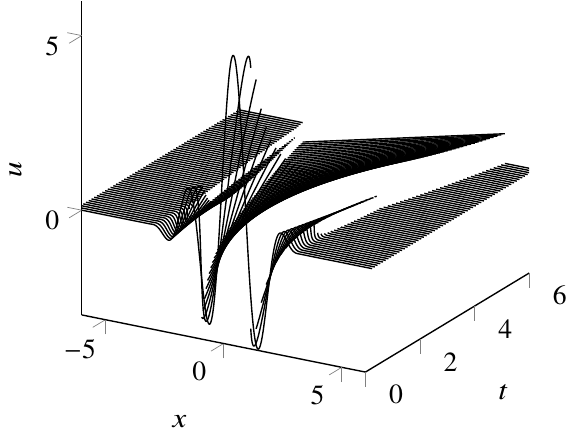}
    \subcaption{}\label{fig:3a}
  \end{minipage}%
  \hfill
  \begin{minipage}[b]{.49\linewidth}
    \centering\includegraphics{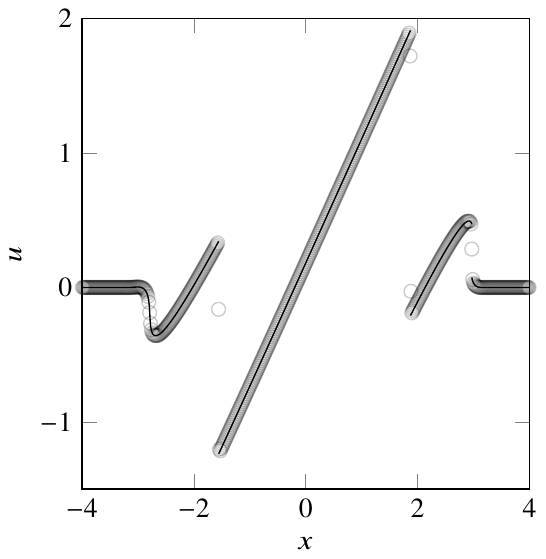}
    \subcaption{}\label{fig:3b}
  \end{minipage}
  \caption{The solution of Burgers' equation for the
    compactly supported initial condition given in
    Example~\ref{ex:burgers_nwave}.  Here (\subref{fig:3a}) is a
    waterfall plot of the solution and (\subref{fig:3b}) is the
    solution for $t=1$ where the solid line is the chebfun constructed
    by the pointwise solution using Implementation~\ref{pdeccl} and
    the transparent circles show the computed solution using a finite
    volume method as described in
    Example~\ref{ex:burgers_nwave}.}\label{fig:burgers_nwave}
\end{figure}
\end{example}

\begin{example}\label{ex:burgers_wiggly}
We solve Burgers' equation with a compactly supported initial condition
with many oscillations
\begin{alignat*}{2}
  \bu_t+{\left( \frac{u^{2}}{2}\right)}_x&=0 && \quad\text{in }\Re\times(0,T)\\
                                \bu&=\begin{cases}
                                  {\left(\sin(x)\right)}^2 + \sin(x^2), & 0 \leq x \leq 14\\
                                  0, & \text{otherwise}
\end{cases}
 && \quad\text{on }\Re\times\{t=0\},
\end{alignat*}
which we use to demonstrate the ability of the proposed algorithm to
find solutions with many shocks.  As in the previous example, we give
the pointwise solution function using Implementation~\ref{pdeccl} to
Chebfun and build piecewise continuous approximations to the solution
shown in Figure~\ref{fig:burgers_wiggly}.  Figure~\ref{fig:burgers_wiggly_img}
further illustrates the complexity of the solution.
 \begin{figure}
  \begin{minipage}[b]{.49\linewidth}
    \centering\includegraphics{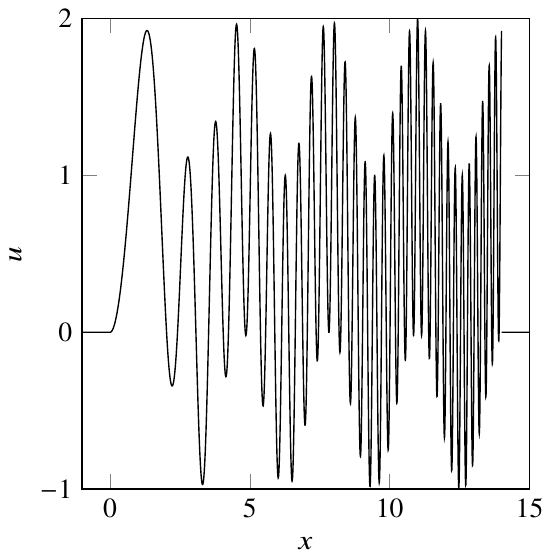}
    \subcaption{$t=0$}\label{fig:4a}
  \end{minipage}%
  \hfill
  \begin{minipage}[b]{.49\linewidth}
    \centering\includegraphics{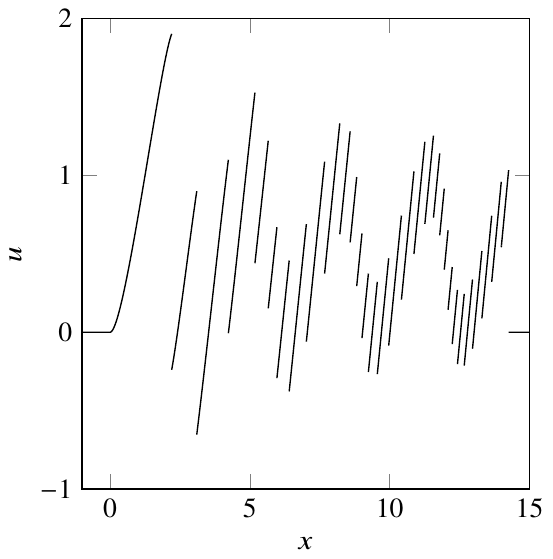}
    \subcaption{$t=\frac12$}\label{fig:4b}
  \end{minipage}
  \begin{minipage}[b]{.49\linewidth}
    \centering\includegraphics{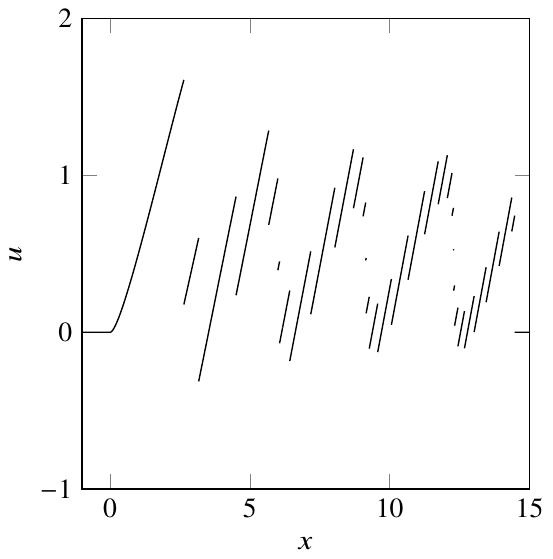}
    \subcaption{$t=1$}\label{fig:4c}
  \end{minipage}%
  \hfill
  \begin{minipage}[b]{.49\linewidth}
    \centering\includegraphics{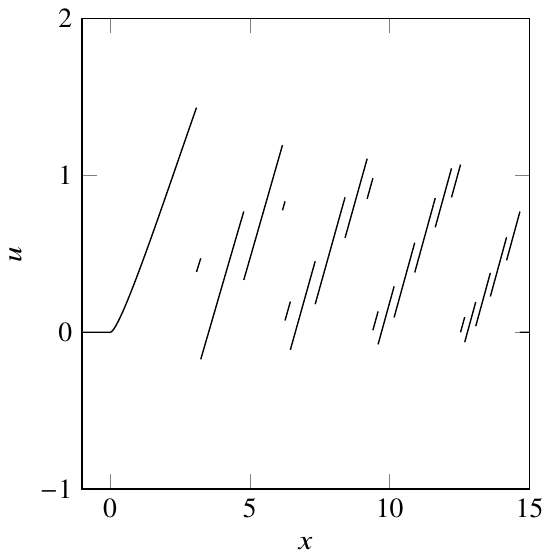}
    \subcaption{$t=\frac32$}\label{fig:4d}
  \end{minipage}

  \caption{Various time instances of the solution of Burgers'
    equation for the compactly supported oscillatory initial condition
    given in Example~\ref{ex:burgers_wiggly}.}\label{fig:burgers_wiggly}
\end{figure}
\begin{figure}
  \centering\includegraphics{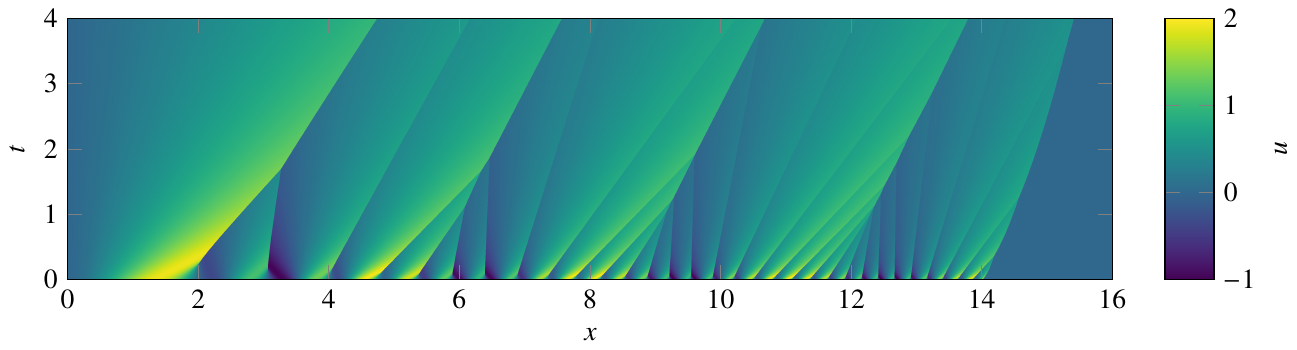}
  \caption{The solution of Burgers' equation for the compactly supported
    oscillatory initial condition given in
    Example~\ref{ex:burgers_wiggly}.}\label{fig:burgers_wiggly_img}
\end{figure}
\end{example}

\begin{example}\label{ex:traffic}
For this example consider a LWR model for traffic flow, named after
\citet{LighthillWhitham1955,Richards1956} and \citet{Richards1956},
\begin{subequations}\label{eq:LWR}
\begin{alignat}{2}
  q_t+{\left( v_{\text{max}}q(1 - q)\right)}_x&=0 && \quad\text{in }\Re\times(0,T), \\
  q&=g && \quad\text{on }\Re\times\{t=0\},
\end{alignat}
where
\begin{alignat}{2}
  g(x) =
  \begin{cases}
    \frac{1}{5} + \frac{4}{5}\exp\left(-\frac{1}{20}{\left(x-\frac{1}{3}\right)}^2\right), & -30 \leq x \leq 30\\
    \frac{1}{5}, & \text{otherwise}
  \end{cases}
 && \quad\text{for }x\in\Re,
\end{alignat}
\end{subequations}
which is an example given in Chapter~11.1 of \citet{LeVeque2002}.
Here $0\le q\le 1$ is the density of the traffic flowing at a maximum
speed $v_{\text{max}}$.  For this example we assume
$v_{\text{max}}=1$.  The flux functions for this conservation law is
concave upward and Algorithm~\ref{al_entropy_solution} requires
a conservation law with a convex downward flux function.  Thus, we let
$u = -q$ which transforms the LWR model~\eqref{eq:LWR} to the convex
conservation law
\begin{alignat*}{2}\label{eq:LWR:convex}
  u_t+{\left( v_{\text{max}}u(1 + u)\right)}_x&=0 && \quad\text{in }\Re\times(0,T), \\
  u&=-g && \quad\text{on }\Re\times\{t=0\}.
\end{alignat*}
As in Example~\ref{ex:burgers_nwave}, we give the pointwise solution
function using Implementation~\ref{pdeccl} to Chebfun and build
piecewise continuous approximations to the solution shown in
Figure~\ref{fig:traffic}.  For verification, this figure also presents
a solution to this problem using a second-order Lax--Wendroff finite
volume method with the van Leer limiter from
Clawpack\footnoteref{foot:clawver}~\citep{clawpack,LeVeque2002} on a
grid from $[-30,30]$ of 500 cells, which took \num{0.0024} seconds to reach
$t=4$.
\begin{figure}
  \begin{minipage}[b]{.49\linewidth}
    \centering\includegraphics{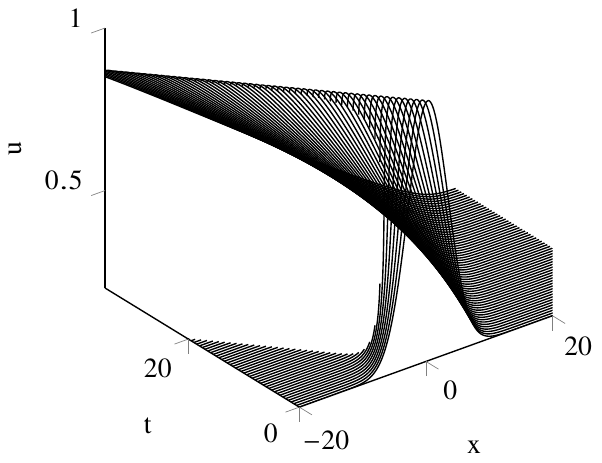}
    \subcaption{}\label{fig:5a}
  \end{minipage}%
  \hfill
  \begin{minipage}[b]{.49\linewidth}
    \centering\includegraphics{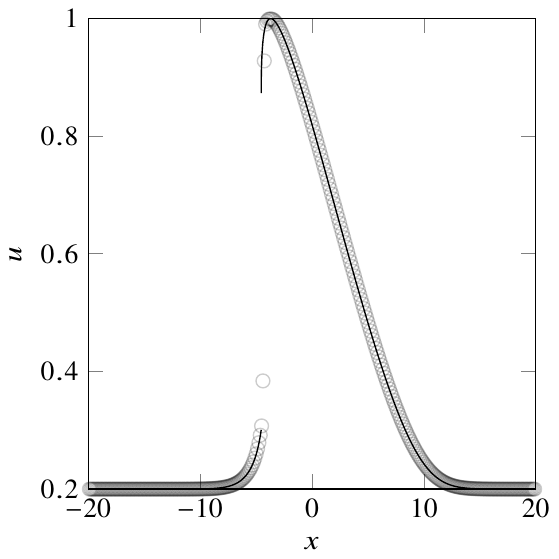}
    \subcaption{}\label{fig:5b}
  \end{minipage}
  \caption{The solution of the LWR equation for the
    initial condition given in Example~\ref{ex:traffic}.  Here
    (\subref{fig:5a}) is a
    waterfall plot of the solution and (\subref{fig:5b}) is the
    solution for $t=4$ where the solid line is the chebfun constructed
    by the pointwise solution using Implementation~\ref{pdeccl} and
    the transparent circles show the computed solution using a finite
    volume method as described in
    Example~\ref{ex:traffic}.}\label{fig:traffic}
\end{figure}
\end{example}

\begin{example}\label{ex:spacedependent}
In this example we consider the following conservation law with a
space-dependent flux function
\begin{subequations}\label{law_spacedependent}
\begin{alignat}{2}
  \bu_t+{\left( \frac{u^{2}-x^{2}}{2}\right)}_x&=0 && \quad\text{in }\Re\times(0,T),\\
                                \bu&=
  \begin{cases}
1, & -1\leq x \leq 0\\
0, & \text{otherwise}
    \end{cases}
 && \quad\text{on }\Re\times\{t=0\}.
\end{alignat}
\end{subequations}
Since the flux function is space-dependent,
Algorithm~\ref{al_entropy_solution} and the theory in
Sections~\ref{sec_3} and~\ref{sec_4} do not apply.
However, we can use Algorithm~\ref{minimum_value_solution}
to compute the unique minimum value solution. Some details required by
the algorithm is given as follows. The associated Lagrangian is
\[L(x,\alpha)=\frac{x^2+\alpha^2}{2}\]
and the Hamiltonian is
\[H(x,p,\alpha)=p\alpha + L(x,\alpha).\]
It is straightforward to derive relationship~\eqref{F2H}
\begin{align*}
\frac{p^{2}-x^2}{2}&=-\min_\alpha H(x,p,\alpha),
\end{align*}
and the associated optimal control~\eqref{alpha}
\begin{align*}
\alpha^\ast &= -p.
\end{align*}
Given any point $(x,t)$ we want to compute $u(x,t)$.
The Hamilton dynamics in \eqref{PMP} and \eqref{PMPbd} have the
following form
\begin{align*}
\dot \bx (r) &= -\bp(r),\\
\dot \bp (r) &= -\bx(r).
\end{align*}
The solution with the initial condition $\bx(0)=x$ is
\begin{align*}
\bx(r)&=(x-C)e^{-r} + Ce^{r}, \\
\bp(r)&=(x-C)e^{-r} - Ce^{r},
\end{align*}
where $C$ is an arbitrary constant.
Denote $\bx(t)$ and $\bp(t)$ by $X$ and $P$, respectively, then we
have
\begin{align*}
X&=(x-C)e^{-t} + Ce^{t}, \\
P&=(x-C)e^{-t} - Ce^{t}.
\end{align*}
Solving for $C$ in terms of $X$ we have
\begin{equation*}
  C=\frac{X-xe^{-t}}{e^{t}-e^{-t}}=2\left(X-xe^{-t}\right)\csch(t)
\end{equation*}
and thus
\begin{equation*}
P=xe^{-t} - \frac{e^{t}+e^{-t}}{e^{t}-e^{-t}}(X-xe^{-t})
 = x\csch(t)-X\coth(t).
\end{equation*}
In Step~I of Algorithm~\ref{minimum_value_solution}, we solve the
PMP~\eqref{PMP} and~\eqref{PMPbd}. It is equivalent to finding $X$
that satisfies at least one set of the following conditions
\begin{subequations}\label{PMPexample}
\begin{align}
&\begin{dcases}
X=\frac{2x}{e^{t}+e^{-t}}=x\sech(t),\\
X \geq 0 \text{ or } X \leq -1,
\end{dcases}\\
  \text{or }&
\begin{dcases}
X=\frac{2x+e^{-t}-e^{t}}{e^{t}+e^{-t}}=x\sech(t)-\tanh(t),\\
-1 < X < 0,
\end{dcases}\\
\text{or } & X = 0,\\
\text{or } & X = -1.
\end{align}
\end{subequations}
The optimal cost in Step~II is
\begin{equation}\label{cost_example}
\begin{split}
J&=\frac{C^2}{2}\left(e^{2t}-1\right)-\frac{{(x-C)}^2}{2}\left(e^{-2t}-1\right) + G(X) \\
&=\frac{x^{2}+X^{2}}{2}\coth(t)-xX\csch(t)+G(X),
\end{split}
\end{equation}
for $t> 0$ where $G$ is a continuous function satisfying  $G^\prime
= \bg$ almost everywhere.

The solution of the conservation law~\eqref{law_spacedependent} using
Algorithm~\ref{al_entropy_solution} is shown in
Figure~\ref{fig_spacedependent_LQR}.  The rarefaction and shock
waves in the solution are clearly shown in the figure.
\begin{figure}
  \begin{minipage}[b]{.49\linewidth}
    \centering\includegraphics{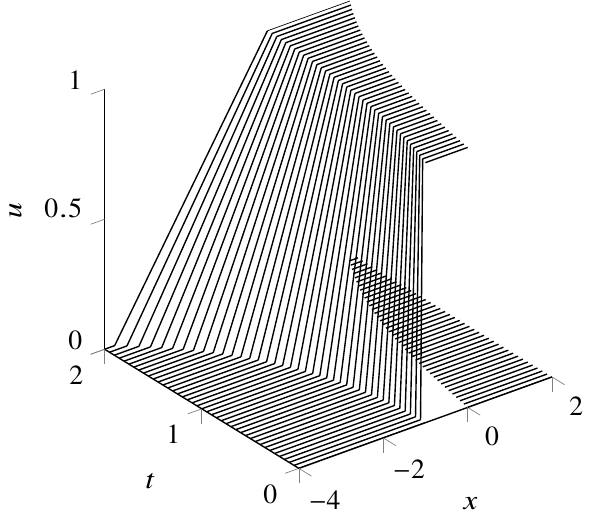}
    \subcaption{}\label{fig:6a}
  \end{minipage}%
  \hfill
  \begin{minipage}[b]{.49\linewidth}
    \centering\includegraphics{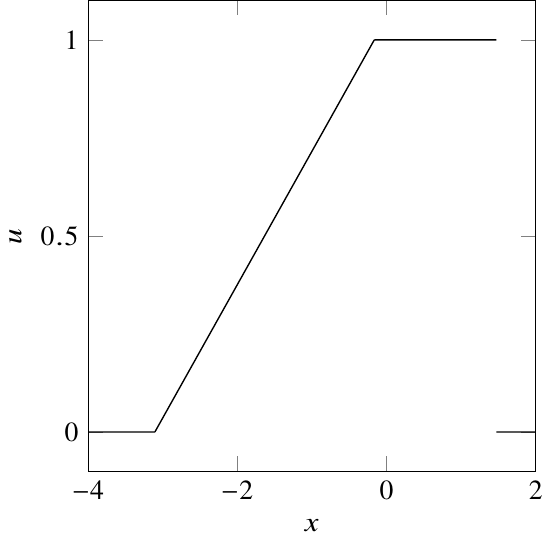}
    \subcaption{}\label{fig:6b}
  \end{minipage}
  \caption{The solution of the conservation law with a space-dependent
    flux given in Example~\ref{ex:spacedependent}.  Here
    (\subref{fig:6a}) is a waterfall plot of the solution and
    (\subref{fig:6b}) is the solution for
    $t=1.8$.}\label{fig_spacedependent_LQR}
\end{figure}
\end{example}

\begin{example}
The follow example is adopted from \citet{zhang}.  Consider the
conservation law with a space-dependent flux
\begin{subequations}\label{law_spacedependent2}
\begin{alignat}{2}
  \bu_t+{\left(\left(1+\frac{\bu}{a}\right)\bu\right)}_x&=0  && \quad\text{in }\Re\times(0,T),\\
  \bu&=-\frac{e^{-10x}}{10} && \quad\text{on }\Re\times\{t=0\},\\
  a&=e^{-10x} && \quad\text{on }\Re.
\end{alignat}
\end{subequations}
It can be shown that the problem has no shocks and the unique solution is
\begin{equation*}
\bu=-\frac{e^{-10x}}{1+9e^{-10t}}.
\end{equation*}
Again, we use Algorithm~\ref{minimum_value_solution} to compute the
unique minimum value solution.  The associated Lagrangian is the
Legendre transform of $F\circ (-1)$,
\begin{equation*}
L(x,\alpha)=\frac{a(x)}{4}{(\alpha +1)}^2.
\end{equation*}
The Hamiltonian is
\begin{equation*}
H(x,p,\alpha)=p\alpha + \frac{a(x)}{4}{(\alpha +1)}^2
\end{equation*}
and the optimal control satisfies
\begin{equation*}
\alpha^\ast =-1-\frac{2p}{a(x)}.
\end{equation*}
Then, it is straightforward to derive the PMP~\eqref{PMP}
\begin{subequations}\label{PMPexample5}
\begin{align*}
\dot \bx(r) &= -1-2e^{10\bx(r)}\bp(r),\\
\dot \bp(r) &= 10 e^{10\bx(r)}{(\bp(r))}^2,\\
\bx(0)&=x,\\
\bp(t)&=-\frac{e^{-10\bx(t)}}{10}.
\end{align*}
\end{subequations}
It is a two-point BVP of differential equations. For the purpose of
testing a numerical implementation of
Algorithm~\ref{minimum_value_solution}, we do not transform
\eqref{PMPexample5} into algebraic equations, although the
differential equations can be explicitly solved. Instead, we use
Chebfun to solve the BVP \eqref{PMPexample5} as described in
\citet{Birkisson2012}.  Due to the uniqueness of solutions, Step~II in
Algorithm~\ref{minimum_value_solution} is unnecessary.  The Chebfun implementation
of damped Newton method is used for the BVP solver with the standard
termination criteria and the standard error tolerance of $10^{-10}$ is
used.  The computation is carried out on a uniform grid of $30\times
30$ points in the region $[0, 1]\times [0.1, 0.6]$ with a maximum
absolute error of $1.33\times 10^{-13}$ is observed.
\end{example}

\section{Conclusions}\label{sec:conclusions}
Using optimal control theory, our study yields an algorithm and
implementation for finding the minimum value solution of a scalar
convex conservation laws, pointwise.  It is proved that in the case of
a space-independent flux function the minimum value solution is the
entropy solution, thus providing a generalization of the Lax--Oleinik
formula.  Numerical results show good agreement of solutions from the
proposed algorithm with that of analytical solutions and the finite
volume method.

\appendix
\section{Example using Implementation~\ref{pdeccl}}\label{a:driver}
As an example of using Implementation~\ref{pdeccl}, we present the
code used in Example~\ref{ex:burgers_const} to generate
Figure~\ref{fig:burgers_const}.
\begin{lstlisting}
Fp = @(u) u;
L  = @(u) u.^2/2;

g = chebfun({0,1,0}, [-1,0,1,3]);
G = cumsum(g);
d = minandmax(g)';
a = g.ends(abs(jump(g, g.ends)) > 10*vscale(g)*eps);

t = 0:0.1:4;
U = cell(size(t));
T = cell(size(t));
u = @(x, t) pdeccl(Fp, L, g, G, d, a, x, t);

T{1} = chebfun(0, domain(g));
U{1} = chebfun(@(x) g(x), domain(g), 'splitting', 'on', 'vectorize');
parfor i = 2:length(t(:))
  T{i} = chebfun(t(i), domain(g));
  U{i} = chebfun(@(x) u(x, t(i)), domain(g), ...
                 'splitting', 'on', 'vectorize');
end

%
x = chebfun('x', domain(g));
for i = 1:length(t(:))
  plot3(x, T{i}, U{i}, 'k'); hold on;
end
hold off;

%
figure;
k = 20;
plot(x, U{k}, 'k');
xlabel('x');
ylabel('u');
title(sprintf('$t = \%f$', T{k}(0)));
\end{lstlisting}
Here we give the point-wise solver \lstinline|pdeccl| to
\lstinline|chebfun| to generate polynomial approximations of the
solution, \lstinline|U|, for various time instances, \lstinline|T|.

\begin{acknowledgements}
 This work was supported in part by AFOSR, NRL, and DARPA\@. Thanks to Maurizio
 Falcone for his insights on fast Legendre--Fenchel transform.
\end{acknowledgements}

\bibliographystyle{spmpscinat}

\end{document}